\documentclass[reqno,12pt]{amsart}
\usepackage{amssymb}
\usepackage{amsmath}
\usepackage{amsfonts}
\usepackage{microtype}
\usepackage[canadian]{babel}
\usepackage{xcolor}
\usepackage[scale=.71]{geometry}

\definecolor{myurlcolor}{rgb}{0,0,0.4}
\definecolor{mycitecolor}{rgb}{0,0.5,0}
\definecolor{myrefcolor}{rgb}{0.5,0,0}
\usepackage[pagebackref,draft=false]{hyperref}
\hypersetup{colorlinks,
linkcolor=myrefcolor,
citecolor=mycitecolor,
urlcolor=myurlcolor}

\usepackage[capitalize]{cleveref}

\theoremstyle{plain}
\newtheorem{thm}{Theorem}
\newtheorem{lem}[thm]{Lemma}

\newtheorem{defn}[thm]{Definition}

\theoremstyle{remark}

\numberwithin{equation}{section}

\crefformat{enumi}{#2#1#3}

\newcommand{\R}{\mathbb{R}}
\newcommand{\Prb}[1]{\mathbf{P}[#1]}
\newcommand{\PM}[1]{\mathcal{P}(#1)}
\newcommand{\beq}{\begin{equation}}
\newcommand{\eeq}{\end{equation}}

\allowdisplaybreaks

\usepackage{todonotes}

\begin{document}

\setlength{\jot}{6pt}

\title{Antisymmetry of the stochastic order\\ on all ordered topological spaces}

\author{Tobias Fritz}

\address{Perimeter Institute for Theoretical Physics, Waterloo, Canada}
\email{tfritz@pitp.ca}



\thanks{\textit{Acknowledgements.} We thank Paolo Perrone for a fruitful and enjoyable ongoing collaboration, without which this result would not have been possible. We also thank the referee for their careful reading and pertinent suggestions. Most of this work was conducted while the author was with the Max Planck Institute for Mathematics in the Sciences.}

\begin{abstract}
In this short note, we prove that the stochastic order of Radon probability measures on any ordered topological space is antisymmetric. This has been known before in various special cases. We give a simple and elementary proof of the general result.
\end{abstract}

\maketitle

\section{Introduction}

Given two real-valued random variables $X$ and $Y$, one says that $X$ is below $Y$ in the \emph{usual stochastic order} if
\[
	\Prb{X \geq c} \, \leq \, \Prb{Y \geq c}
\]
for all $c \in \R$. Intuitively, if $X$ and $Y$ describe the return distributions of e.g.~two financial assets, then $Y$ is clearly at least as good as $X$ is. Since this order relation does not depend on the joint distribution, it really is an ordering on the set of probability measures on $\R$. It is also known under other names, such as \emph{(first-order) stochastic dominance}; in this note, we will simply speak of the \emph{stochastic order}. It is clear that the stochastic order is a preorder relation, i.e.~it is reflexive and transitive. It is also easy to see that it is an \emph{antisymmetric} relation on probability measures: the above inequality represents an inequality between cumulative distribution functions, and a probability measure is uniquely determined by its CDF. Hence if $X$ is stochastically dominated by $Y$ and conversely $Y$ by $X$, then their distributions must be the same.

The stochastic order has been generalized long ago to probability measures on ordered topological spaces~\cite{strassen,edwards,kellerer}. More recent investigations have focused on ordered metric spaces, exploiting and investigating interactions between the order and the metric~\cite{lawson,hll,ours_ordered}. Again it is not hard to see that the stochastic order is reflexive and transitive---although the latter is not necessarily obvious, depending on which characterization one uses as the definition (Theorem~\ref{chars}).

However, proving antisymmetry of the stochastic order takes a bit more work. Edwards~\cite[p.~59/71]{edwards} noted that it holds for compact spaces due to the Stone--Weierstra\ss{} theorem, and more generally for completely regular ordered spaces in the sense of Nachbin~\cite{nachbin} thanks to the existence of Nachbin compactifications. Hiai, Lawson and Lim proved recently that it holds for Radon probability measures of finite first moment on certain types of cones in Banach spaces~\cite[Theorem~4.3]{hll}. Although this is a special case of Edwards' result, their proof is different. Subsequently, we proved that antisymmetry also holds for finite first moment measures on a large class of metric spaces which we call \emph{L-ordered}~\cite{ours_ordered}, where L-orderedness is a natural compatibility condition between order and metric slightly stronger than closedness of the order. We also showed that the class of L-ordered spaces coincides with the class of ordered metric spaces which can be embedded into ordered Banach spaces. Since not every ordered Banach space is completely regular ordered, this antisymmetry result is not covered by Edwards'.

In this note, we prove that antisymmetry indeed holds for all Radon probability measures on all ordered topological spaces (Theorem~\ref{general_antisym}). The only required compatibility condition between the partial order and the topology is that the order must be closed.

\section{Ordered topological spaces and stochastic order}

Here we set up the necessary definitions and state some known equivalent characterizations of the stochastic order. In Section~\ref{main}, we then prove the antisymmetry of the stochastic order.

Recall the following standard definition:

\begin{defn}
	An \emph{ordered topological space} is a triple $(X,\mathcal{T},\le)$ where $(X,\mathcal{T})$ is a topological space and $(X,\le)$ is a partially ordered set, and such that the set of all ordered pairs
\beq
\label{ord_pairs}
	\{ (x,y) \in X \times X \mid x \le y \}
\eeq
is closed in $X \times X$.
\end{defn}

In the following, we also use the shorthand notation $\{\le\}$ for the set~\eqref{ord_pairs}. It is easy to see that, $X$ is necessarily Hausdorff since the set $\{\ge\} \cap \{\le\} \subseteq X \times X$ is closed by closedness of the order, and equal to the diagonal as the order is assumed partial~\cite[Proposition~1.2]{nachbin}. A set $S \subseteq X$ is an \emph{upper set} if $x\in S$ and $x\le y$ implies $y\in S$. Every set $S$ generates an upper set $\uparrow\! S$, the smallest upper set which contains $S$.

The following standard fact~\cite[Proposition~3.4]{nachbin} will be useful in our proof.

\begin{lem}
	Let $X$ be an ordered topological space. If $C \subseteq X$ is compact, then $\,\uparrow\! C$ is closed.
	\label{lemma}
\end{lem}

Writing $\PM{X}$ for the set of Radon probability measures on $X$, the stochastic order on $\PM{X}$ is defined by either one of the following equivalent characterizations.

\begin{thm}[Strassen, Kellerer, Edwards]
\label{chars}
Let $X$ be an ordered topological space. For $p,q\in\PM{X}$, the following are equivalent:
\begin{enumerate}
\item\label{upsets_closed} $p(C) \leq q(C)$ for every closed upper set $C\subseteq X$;
\item\label{upsets_open} $p(U) \leq q(U)$ for every open upper set $U\subseteq X$;
\item\label{monotones} For every monotone and bounded lower semi-continuous function $f : X \to \R$, we have
\[
	\int_X f(x) \, dp(x) \le \int_X f(x) \, dq(x).
\]
\item\label{joint} There is $r \in \PM{X\times X}$ supported on $\{\le\}$ such that the marginals of $r$ are $p$ and $q$, respectively.
\end{enumerate}
\end{thm}

Here, the equivalence of~\ref{upsets_closed} and~\ref{joint} is due to Strassen's theorem~\cite[Theorem~11]{strassen} for ordered Polish spaces, and due to Kellerer~\cite[Proposition~3.12]{kellerer} for all ordered topological spaces. The equivalence with~\ref{upsets_open} follows from this upon applying Kellerer's result to the opposite order and using $p(U) = 1 - p(\bar{U})$. For the equivalence with~\ref{monotones}, see Edwards~\cite[Theorem~7.1]{edwards}.

In the following, we will use~\ref{joint} as the most convenient characterization of the stochastic order.

\section{Main result}
\label{main}

\begin{thm}
For any ordered topological space $X$, the stochastic order on the Radon probability measures on $X$ is antisymmetric. 
\label{general_antisym}
\end{thm}

\begin{proof}

	We need to assume that we have Radon probability measures $p$ and $q$ with $p \leq q \leq p$, where these order relations are witnessed, respectively, by joint Radon probability measures $r$ and $s$ supported on the set $\{\le\} \subseteq X\times X$. 
We will prove that $r$ is supported on the diagonal. This then implies that for every Borel set $A\subseteq X$, we have
\[
	p(A) = r(A\times X) = r(A\times A) = r(X\times A) = q(A),
\]
and hence the claim $p = q$ follows. Here, the second and third equation hold because $r$ is supported on the diagonal, which gives $r(A \times (X\setminus A)) = 0 = r((X\setminus A) \times A)$.

In order to prove that $r$ is supported on the diagonal, we need to show that for any $x,y\in X$ with $x < y$ there is a neighbourhood $U \times V \ni (x,y)$ with $r(U \times V) = 0$. Indeed by closedness of the order relation, we can find neighbourhoods $U \ni x$ and $V \ni y$ with $x' \not\ge y'$ for all $x' \in U$ and $y' \in V$. In other words, $U \,\cap \uparrow\! V = \emptyset$. Then in order to show that $r(U \times V) = 0$, by inner regularity it is enough to prove $r(C) = 0$ for all compact $C \subseteq U \times V$. Taking $D \subseteq X$ and $E \subseteq Y$ to be the projections of $C$ onto the two factors, $D \times E$ is a compact set with $C \subseteq D \times E \subseteq U \times V$. It is therefore enough to prove $r(D \times E) = 0$, using the fact that $D \,\cap \uparrow\! E = \emptyset$.

Now by \Cref{lemma}, the set $\uparrow\! E$ is closed and in particular Borel. We can therefore compute
\begin{align}
	\begin{split}
		\label{ineq}
		p(\uparrow\! E) & = r(\uparrow\! E\times X) = r(\uparrow\! E \, \times \uparrow\! E) \\
		& \le r(\uparrow\! E \, \times \uparrow\! E) + r(D \times E) \le r(X \times \uparrow\! E) = q(\uparrow\! E).
	\end{split}
\end{align}
Here, the second equation is by upper closedness of $\uparrow\! E$ and the assumption that $r$ is supported on $\{\leq\}$; while the second inequality holds because of $D \, \cap \uparrow\! E = \emptyset$. On the other hand, using the other joint $s$ gives
\begin{align*}
	q(\uparrow\! E) & = s(\uparrow\! E \times X) = s(\uparrow\! E \, \times \uparrow\! E) \\
	& \leq s(X\times \uparrow\! E) = p(\uparrow\! E).
\end{align*}
Therefore the inequalities in~\eqref{ineq} are all equalities, giving in particular the claimed $r(D \times E) = 0$. Hence $r$ is indeed supported on the diagonal.
\end{proof}

The argument given in this proof is really the obvious argument that one would make in very simple cases, e.g.~when proving that the stochastic order is antisymmetric on any finite partially ordered set.

\bibliographystyle{plain}
\bibliography{catprob}

\end{document}